\documentclass[a4paper]{amsart}

\usepackage[greek,english]{babel}
\usepackage{microtype}

\usepackage[T1]{fontenc}
\usepackage{libertinus}
\usepackage{amsmath}
\usepackage{amssymb}
\usepackage{mathtools}

\usepackage[
  libertinus,
  noto,
  vvarbb,
  amsthm,
  nott, 
]{newtx}

\usepackage[scale=0.95,varqu,varl]{inconsolata}
  
\usepackage{xfrac}
\usepackage{xcolor}

\usepackage{tikz}
\usetikzlibrary{arrows}
\usetikzlibrary{arrows.meta}

\usepackage[pdftex,
            pdfauthor={Michael Eisermann},
            pdftitle={Formalizing Pick's Theorem, efficiently}, 
            pdfstartview=FitH,
            colorlinks=true,
            anchorcolor=black, 
            linkcolor=black, 
            urlcolor=black,
            citecolor=black, 
            pagecolor=black, 
            filecolor=black, 
            menucolor=black, 
            ]{hyperref}

\usepackage[a4paper]{geometry}
\title{Formalizing Pick's Theorem, efficiently} 

\author{Michael Eisermann}

\address{Institut für Geometrie und Topologie, Universität Stuttgart, Germany}
\email{Michael.Eisermann@mathematik.uni-stuttgart.de}

\author{Elli zu Neblberg}

\date{Burg Liebenzell, first draft September 2024, this version compiled \today}

\numberwithin{equation}{section}

\theoremstyle{plain}
  \newtheorem{theorem}{Theorem}[section]
  \newtheorem{lemma}[theorem]{Lemma}

\theoremstyle{definition}
  \newtheorem{definition}[theorem]{Definition}

  \newtheorem{remark}[theorem]{Remark}
  \newtheorem{example}[theorem]{Example}
  
\newcommand{\N}{\mathbb{N}}
\newcommand{\Z}{\mathbb{Z}}
\newcommand{\Q}{\mathbb{Q}}
\newcommand{\R}{\mathbb{R}}
\newcommand{\K}{\mathbb{K}}

\newcommand{\ii}[2]{\mathopen[ #1, #2 \mathclose]}

\newcommand{\ei}[2]{\mathopen] #1, #2 \mathclose]}
\newcommand{\ee}[2]{\mathopen] #1, #2 \mathclose[}

\newcommand{\fade}{\color{black!50}}

\newcommand{\minus}{\smallsetminus}
\newcommand{\abs}[1]{\lvert #1 \rvert}

\DeclareMathOperator{\sign}{sign}
\DeclareMathOperator{\vol}{vol}
\DeclareMathOperator{\supp}{supp}

\newcommand{\area}{\mathrm{area}}
\newcommand{\Area}{\mathrm{Area}}
\newcommand{\ang}{\mathrm{ang}}
\newcommand{\Ang}{\mathrm{Ang}}
\newcommand{\dang}{\mathrm{dang}}
\newcommand{\Dang}{\mathrm{Dang}}
\newcommand{\welp}{\mathrm{welp}}
\newcommand{\Welp}{\mathrm{Welp}}

\def\middlevert{\mskip-0mu\mid\mskip-0mu}
\begingroup
\catcode`\|=\active
\gdef\set#1{\begingroup \mathcode`\|32768 \let|\middlevert \{\mskip2mu #1 \mskip2mu\} \endgroup}
\endgroup

\begin{document} 

\begin{abstract}
  Pick's astonishing theorem explains how to obtain 
  the area of any integer polygon by counting lattice points. 
  It is a notoriously difficult challenge 
  to translate the geometric statement and intuitive reasoning
  into a formal statement and rigorous proof. 
  We transform the beautiful geometry into equally elegant algebra,
  and then implement the algebraic proof in Lean.  
\end{abstract}

\maketitle

\begin{center}
  \small\it
  “Algebra is but written geometry, geometry is but figured algebra.”
  \rm (Sophie Germain, 1776--1831)
\end{center}

\section{Introduction: Pick's wonderful theorem} \label{sec:Introduction}

A \emph{polygon} $P = (p_0,p_1,\ldots,p_n)$ is 
a finite sequence of points $p_i = (x_i,y_i) \in \R^2$.
It defines the corresponding path $\gamma \colon [0,1] \to \R^2$
by piecewise linear interpolation of the given vertices
$\gamma(i/n) = p_i$ for $i=0,1,\dots,n$.
We call $P$ \emph{closed} if $p_0 = p_n$, and
\emph{simply closed} if $\gamma(s) = \gamma(t)$
only holds for $s = t$ or $\{s,t\} = \{0,1\}$.
In this case Jordan's theorem applies:
The polygonal curve $C = \gamma([0,1]) \subset \R^2$
separates the plane in two connected open sets $A$ and $B$,
so $\R^2 = A \sqcup B \sqcup C$, where the \emph{exterior region} $A$
is unbounded and the \emph{interior region} $B$ is bounded.
Their common boundary is the curve $C$, so their closures
are $\bar{A} = A \cup C$ and $\bar{B} = B \cup C$.

\begin{figure}[ht]
  \begin{tikzpicture}[x=5mm, y=5mm, font=\footnotesize,
      rounded corners=0.2pt, inner sep=2pt, align= left]
    \setlength{\baselineskip}{5mm}
    \draw[color=black!50, step=1] (-1,-1) grid (21,7);
    \begin{scope}[shift={(0,0)}]
      \draw[line width=1pt, blue, fill=yellow!50, fill opacity=0.3]
      (0,2) -- (2,0) -- (5,0) -- (5,2) -- (4,3) -- (5,5) --
      (2,6) -- (0,3) -- (2,5) -- (4,1) -- (3,1) -- (2,3) -- cycle;
      \draw[inner sep=1pt, font=\tiny\fade\mathstrut] 
      (2,3) node[above left]{$p_{11}$}
      (0,2) node[below left]{$p_0$}
      (2,0) node[below left]{$p_1$}
      (5,0) node[below left]{$p_2$};
      \foreach \x/\y in { 1/2, 2/1, 2/2, 3/4, 3/5, 4/2, 4/4, 4/5 }{
        \draw[draw=black, thin, fill=blue!70!white] (\x,\y) circle[radius=1.5pt];
      }
      \foreach \x/\y in { 0/2, 1/1, 2/0, 3/0, 4/0, 5/0, 5/1, 5/2,
        4/3, 5/5, 2/6, 0/3, 1/4, 2/5, 3/3, 4/1, 3/1, 2/3 }{
        \draw[draw=black, thin, fill=blue] (\x,\y) circle[radius=1.5pt];
      }
      \foreach \x/\y in { 0/2, 2/0, 5/0, 5/2, 4/3, 5/5, 2/6, 0/3, 2/5, 4/1, 3/1, 2/3 }{
        \draw[draw=black, thin, fill=blue!70!black] (\x,\y) circle[radius=1.5pt];
      }
      \draw (0,6) node[below right] {(a)};
    \end{scope}
    \begin{scope}[shift={(6,0)}]
      \draw (0,6) node[below right]{(b) simple \\ but not closed};
      \draw[line width=1pt, blue] (1,0) -- (4,1) -- (0,3) -- (3,4);
      \draw[line width=1pt, blue, dashed ] (3,4) -- (1,0) node[pos=0.4, above left]{?};
      \foreach \x/\y in { 1/0, 4/1, 0/3, 3/4 }{
        \draw[draw=black, thin, fill=blue!70!black] (\x,\y) circle[radius=1.5pt];
      }      
    \end{scope}
    \begin{scope}[shift={(11,0)}]
      \draw (0,6) node[below right]{(c) closed \\ but not simple};
      \draw[draw=none, fill=teal!30, fill opacity=0.3]
      (1,0) -- (4,1) -- (2,2) -- cycle;
      \draw[draw=none, fill=orange!40, fill opacity=0.3]
      (0,3) -- (3,4) -- (2,2) -- cycle;
      \draw[line width=1pt, blue] (1,0) -- (4,1) -- (0,3) -- (3,4) -- cycle;
      \foreach \x/\y in { 1/0, 4/1, 0/3, 3/4 }{
        \draw[draw=black, thin, fill=blue!70!black] (\x,\y) circle[radius=1.5pt];
      }      
    \end{scope}
    \begin{scope}[shift={(16,0)}]
      \draw (0,6) node[below right, ]{(d) simply closed \\ but not integer};
      \draw[line width=1pt, blue, fill=purple!30, fill opacity=0.3]
      (0,2.5) -- (2,0) -- (3,2) -- (2,4) -- cycle;
      \foreach \x/\y in { 2/0, 3/2, 2/4 }{
        \draw[draw=black, thin, fill=blue!70!black] (\x,\y) circle[radius=1.5pt];
      }
      \draw[draw=black, thin, fill=purple] (0.0,2.5) circle[radius=1.5pt];
    \end{scope}
    \draw (23,3) node{\includegraphics[height=30mm]{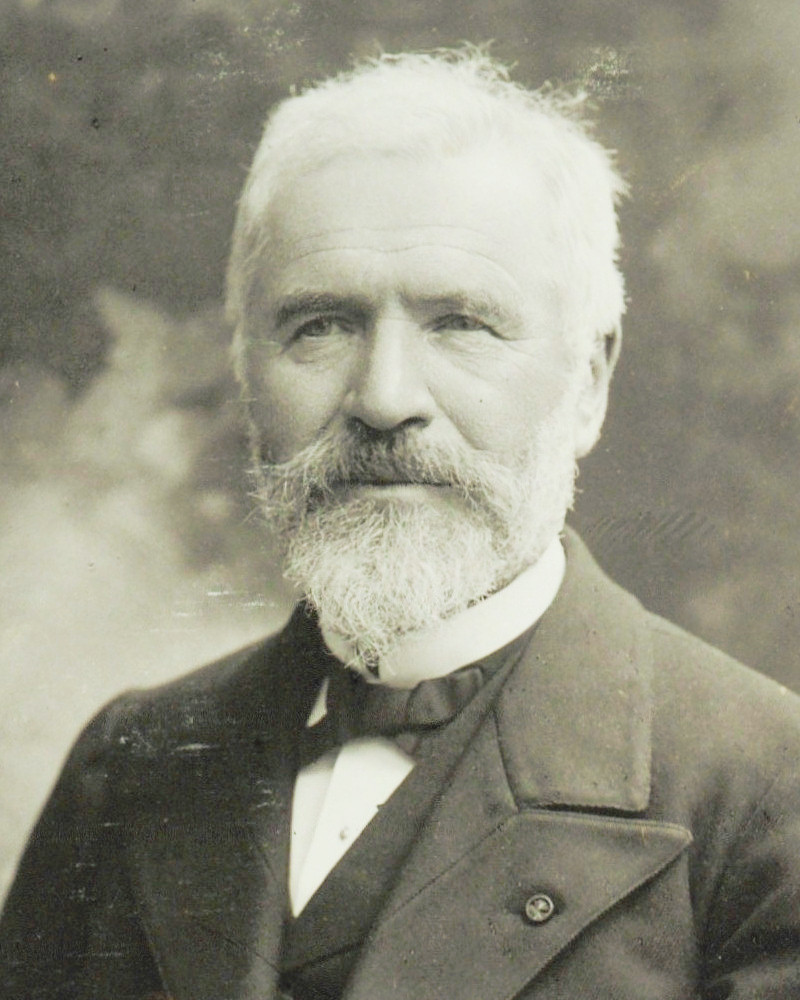}};
    \draw (23,0) node[below, font=\footnotesize]{Camille Jordan};
  \end{tikzpicture}
\end{figure}

Let $\vol_2(B)$ denote the enclosed area.
On the other hand, we can count the number of lattice points,
$I := \abs{ \Z^2 \cap B }$ in the interior and
$J := \abs{ \Z^2 \cap C }$ on the boundary $C = \partial B = \partial A$.

Here the magic of Pick's theorem happens:

\begin{theorem}[Georg Pick 1899 \cite{Pick:1899}] \label{thm:PickCombinatorial}
  Let $P = (p_0,p_1,\ldots,p_n=p_0)$
  be a simply closed polygon with integer vertices $p_i \in \Z^2$.
  Then $\vol_2(B) = I + J/2 - 1$.
\end{theorem}

\begin{figure}[ht]
  \begin{tikzpicture}[x=5mm, y=5mm, rounded corners=0.1pt]
    \begin{scope}[shift={(0.0,-2)}]
      \draw[color=black!50, step=1] (-1,-1) grid (6,5);
      \draw[line width=1pt, blue, fill=yellow!50, fill opacity=0.3] (0,0) rectangle (5,4);
      \draw[inner sep=1pt, font=\small\fade\mathstrut] 
      (0,0) node[below right]{$0$}
      (5,0) node[below left]{$a$}
      (0,0) node[above left]{$0$}
      (0,4) node[below left]{$b$};
      \clip (0,0) rectangle (5,4);
      \foreach \x in {0,...,5}{
        \foreach \y in {0,...,4}{
          \draw[draw=none, fill=blue, fill opacity=0.5] (\x,\y) circle[radius=1mm];
        }
      }
    \end{scope}
    \begin{scope}[shift={(9.5,0)}]
      \draw[color=black!50, step=1] (-3,-3) grid (3,3);
      \draw[line width=1pt, teal, fill=yellow!50, fill opacity=0.3]
      (+3,0) -- (0,+3) -- (-3,0) -- (0,-3) -- cycle;
      \draw[draw=none, inner sep=1pt, font=\small\fade\mathstrut]
      (-3,0) -- (0,+3) node[midway, sloped, above] {$3\sqrt{2}$}; 
      \clip (+3,0) -- (0,+3) -- (-3,0) -- (0,-3) -- cycle;
      \foreach \x in {-3,...,3}{
        \foreach \y in {-3,...,3}{
          \draw[draw=none, fill=teal, fill opacity=0.5] (\x,\y) circle[radius=1mm];
        }
      }
    \end{scope}
    \begin{scope}[shift={(16,0)}, x=2.5mm, y=2.5mm]
      \draw[color=black!50, step=1] (-6,-6) grid (6,6); 
      \draw[line width=1pt, orange, fill=yellow!50, fill opacity=0.3]
      (+1,+0) -- (+6,+1) -- (+3,+1) -- (+5,+2) -- (+2,+1) -- (+5,+3) -- (+3,+2) -- (+6,+5) --
      (+1,+1) -- (+5,+6) -- (+2,+3) -- (+3,+5) -- (+1,+2) -- (+2,+5) -- (+1,+3) -- (+1,+6) --
      (-0,+1) -- (-1,+6) -- (-1,+3) -- (-2,+5) -- (-1,+2) -- (-3,+5) -- (-2,+3) -- (-5,+6) --
      (-1,+1) -- (-6,+5) -- (-3,+2) -- (-5,+3) -- (-2,+1) -- (-5,+2) -- (-3,+1) -- (-6,+1) --
      (-1,-0) -- (-6,-1) -- (-3,-1) -- (-5,-2) -- (-2,-1) -- (-5,-3) -- (-3,-2) -- (-6,-5) --
      (-1,-1) -- (-5,-6) -- (-2,-3) -- (-3,-5) -- (-1,-2) -- (-2,-5) -- (-1,-3) -- (-1,-6) --
      (+0,-1) -- (+1,-6) -- (+1,-3) -- (+2,-5) -- (+1,-2) -- (+3,-5) -- (+2,-3) -- (+5,-6) --
      (+1,-1) -- (+6,-5) -- (+3,-2) -- (+5,-3) -- (+2,-1) -- (+5,-2) -- (+3,-1) -- (+6,-1) --
      cycle;
      \draw[draw=none, fill=orange, fill opacity=0.5] (0,0) circle[radius=1mm];
    \end{scope}
    \draw (22, 0) node{\includegraphics[height=30mm]{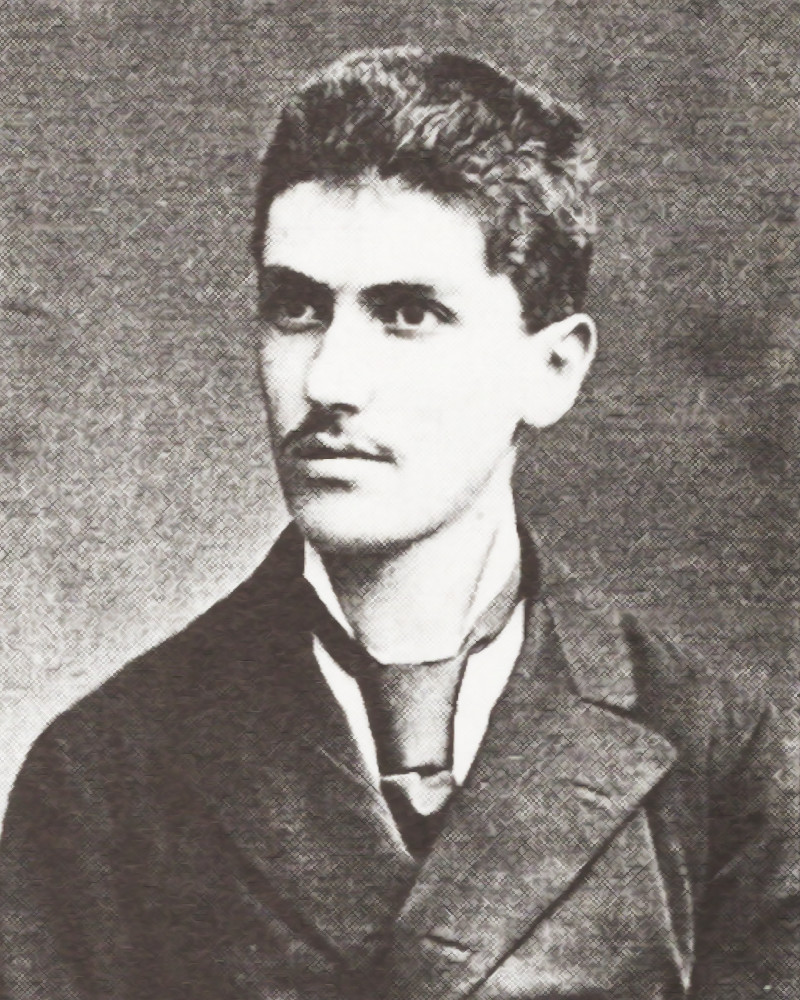}};
    \draw (22,-3) node[below, font=\footnotesize]{Georg Pick};
  \end{tikzpicture}
  \caption{(a) rectangle, (b) oblique square, (c) Farey sunburst $F_6$} 
  \label{fig:PickExamples}
\end{figure}
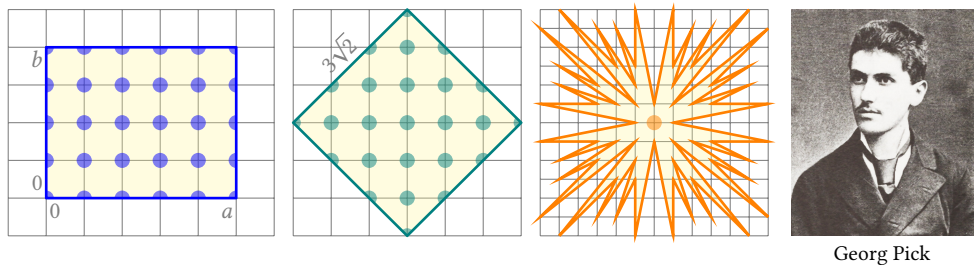

\begin{example}
  (a) As a simple illustration, consider 
  the rectangle $R = [0,a] \times [0,b]$ with $a,b \in \N_{\ge1}$.
  We find $I = (a-1) (b-1)$ and $J = 2 a + 2 b$,
  which nicely adds up to $\vol_2(R) = a b$.

  (b) For the oblique square $Q$ of \autoref{fig:PickExamples}(b) 
  we find $I = 13$ and $J = 12$, whence $\vol_2(Q) = 18$.
  By Pick's theorem we can measure the area
  or count lattice points, whichever is simpler.
  
  (c) We calculate the area of the Farey sunburst $F_6$ shown in \autoref{fig:PickExamples}(c).
  Its $64$ vertices are given by $(x,y) \in \{-6,\dots,6\}^2$ with $\gcd(x,y) = 1$.
  Closer inspection reveals $32$ further boundary points. Its only inner point is $(0,0)$.
  We conclude that $\vol_2(F_6) = 1 + 96/2 - 1 = 48$.
\end{example}

\begin{remark}
  It is essential that the polygon $P$ be given by integer vertices.
  The rectangle $R = [0,a] \times [1/4,3/4]$, for example,
  has area $a/2$ but dos not contain any lattice point.
\end{remark}

\begin{remark}
  Pick's theorem is special to the plane $\R^2$.
  No such simple formula can hold in higher dimensions:
  The Reeve tetrahedron $T_r \subset \R^3$ is the convex
  hull of the four integer vertices $(0,0,0)$, $(1,0,0)$, $(0,1,0)$
  and $(1,1,r)$ with $r \in \N$.
  It has arbitrarily large volume $\vol_3(T_r) = r/6$,
  yet contains no further lattice points.
\end{remark}

\begin{remark}
  According to the formula $\vol_2(B) = I + J/2 - 1$,
  the area of any simply closed integer polygon
  is always integer or half integer,
  in brief $\vol_2(B) \in \frac{1}{2} \Z$.

  Here is a nice application: 
  Can you construct an equilateral triangle $\Delta \subset \R^2$
  with three integer vertices $(0,0)$, $(a,b)$ and $(c,d)$?
  No!  Its side length $\ell$ would satisfy $\ell^2 = a^2 + b^2 \in \N$ by Pythagoras,
  so its area $\vol_2(\Delta) = \sqrt{3}/4 \cdot \ell^2$
  cannot be in $\frac{1}{2} \Z$, since $\sqrt{3}$ is irrational.
\end{remark}

\section{How to formalize Pick's theorem?}

We work over an ordered field $(\K,+,\cdot,\le)$.
The traditional choice is the field $\R$
of real numbers, which we consider first.
It turns out, however, that the field $\Q$ of rational numbers suffices,
and moreover is more convenient for computer implementations.
By abstracting both these primary examples
to ordered fields we cover all cases simultaneously.

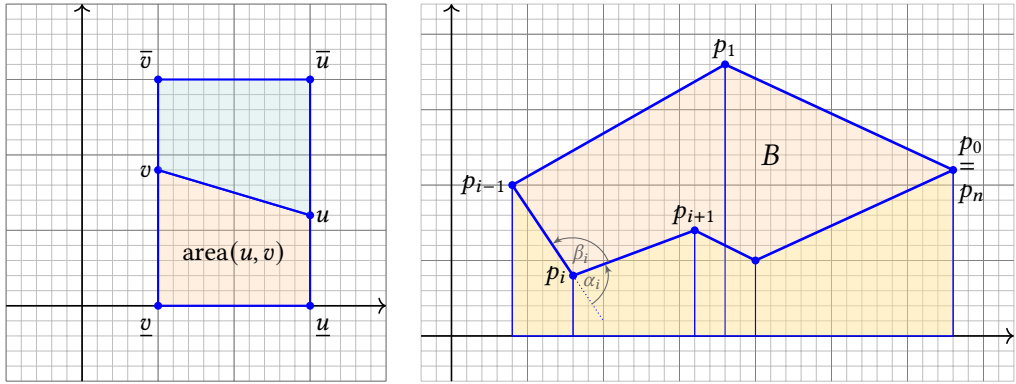
\begin{figure}[ht]
  \begin{tikzpicture}[x=10mm, y=10mm, font=\mathstrut, inner sep=2pt]
    \draw[style= help lines, very thin, color=lightgray, step=0.2] (-1,-1) grid (4,4);
    \draw[style= help lines, very thin, color=gray, step=1] (-1,-1) grid (4,4);
    \draw[semithick, ->] (-1,0) -- (4,0); 
    \draw[semithick, ->] (0,-1) -- (0,4); 
    \filldraw[draw=blue, thick, fill=orange!40, fill opacity=0.3]
    (3.0,1.2) -- (1.0,1.8) -- (1.0,0.0) -- (3.0,0.0) -- cycle;
    \draw[draw=none, fill=blue, radius=1.5pt]
    (3.0,1.2) circle node[right] {$u$}
    (1.0,1.8) circle node[left] {$v$}
    (1.0,0.0) circle node[below left] {$\underline{v}$}
    (3.0,0.0) circle node[below right] {$\underline{u}$};
    \filldraw[draw=blue, thick, fill=teal!30, fill opacity=0.3]
    (3.0,1.2) -- (3.0,3.0) -- (1.0,3.0) -- (1.0,1.8) -- cycle;
    \draw[draw=none, fill=blue, radius=1.5pt]
    (1.0,3.0) circle node[above left] {$\overline{v}$}
    (3.0,3.0) circle node[above right] {$\overline{u}$};
    \draw (2.0,0.7) node{$\area(u,v)$};
  \end{tikzpicture}
  \quad
  \begin{tikzpicture}[x=10mm, y=10mm, rounded corners=0.2pt]
    \draw[style= help lines, very thin, color=lightgray, step=0.2] (-0.4,-0.6) grid (7.4,4.4);
    \draw[style= help lines, very thin, color=gray, step=1.0] (-0.4,-0.6) grid (7.4,4.4);
    \draw[semithick, ->] (-0.4,0.0) -- (7.4,0.0);
    \draw[semithick, ->] (0.0,-0.6) -- (0.0,4.4);

    \coordinate (P0) at (6.6,2.2);
    \coordinate (P1) at (3.6,3.6);
    \coordinate (P2) at (0.8,2.0);
    \coordinate (P3) at (1.6,0.8);
    \coordinate (P4) at (3.2,1.4);
    \coordinate (P5) at (4.0,1.0);
    
    \coordinate (Q0) at (6.6,0.0);
    \coordinate (Q1) at (3.6,0.0);
    \coordinate (Q2) at (0.8,0.0);
    \coordinate (Q3) at (1.6,0.0);
    \coordinate (Q4) at (3.2,0.0);
    \coordinate (Q5) at (4.0,0.0);
    
    \draw[blue, fill=orange!40, fill opacity=0.3]
    (Q0) -- (P0) -- (P1) -- (Q1) -- cycle;
    \draw[blue, fill=orange!40, fill opacity=0.3]
    (Q1) -- (P1) -- (P2) -- (Q2) -- cycle;
    \draw[blue, fill=yellow!50, fill opacity=0.3]
    (Q2) -- (P2) -- (P3) -- (Q3) -- cycle;
    \draw[blue, fill=yellow!50, fill opacity=0.3]
    (Q3) -- (P3) -- (P4) -- (Q4) -- cycle;
    \draw[blue, fill=yellow!50, fill opacity=0.3]
    (Q4) -- (P4) -- (P5) -- (Q5) -- cycle;
    \draw[blue, fill=yellow!50, fill opacity=0.3]
    (Q5) -- (P5) -- (P0) -- (Q0) -- cycle;

    \draw[blue, line width=1pt]
    (P0) -- (P1) -- (P2) -- (P3) -- (P4) -- (P5) -- cycle;
    \draw[draw=none, fill=blue, radius=1.5pt, inner sep=2pt]
    (P0) circle node[right, align=left] {$p_0$ \\[-3pt] $=$ \\[-3pt] $p_n$}
    (P1) circle[] node[above] {$p_1$}
    (P2) circle[] node[left] {$p_{i-1}$}
    (P3) circle[] node[left] {$p_i$}
    (P4) circle[] node[above] {$p_{i+1}$}
    (P5) circle[]; 

    \draw[blue, densely dotted] (P3) -- +(+0.4,-0.6);
    \draw[draw=black!60, -latex'] (P3) +(-57:0.45) arc (-57:+21:0.45);
    \draw (P3) +(-20:0.27) node[black!60, scale=0.8]{$\alpha_i$};
    \draw[draw=black!60, -latex'] (P3) +(+21:0.50) arc (+21:124:0.50);
    \draw (P3) +(70:0.30) node[black!60, scale=0.8]{$\beta_i$};

    \draw (4.2,2.4) node[scale=1.2]{$B$};

  \end{tikzpicture}
  \caption{Oriented area of a trapezoid and of a polygon}
  \label{fig:AreaAngle}
\end{figure}

\begin{definition}
  Given two points $u,v \in \K^2$ we have the (oriented) area of the trapezoid,
  \begin{align}
    \area(u,v) := \frac{1}{2} (u_1 - v_1) (u_2 + v_2) .
  \end{align}
  
  For every polygon $P = (p_0,p_1,\ldots,p_n)$ in $\K^2$
  we thus define its area to be 
  \begin{align}
    \label{eq:Area}
    \Area(P) := \sum_{i=1}^{n} \area(p_{i-1},p_i) .
  \end{align}
\end{definition}

\begin{remark}
  For a simply closed polygon in $\R^2$, this yields $\abs{\Area(P)} = \vol_2(B) > 0$.
  
  If $\Area(P) > 0$, we call our polygon $P$ \emph{positively oriented}.
  Otherwise, if $\Area(P) < 0$, we reverse $P = (p_0,p_1,\ldots,p_n)$
  to $P' = (p_n,\ldots,p_1,p_0)$.  This does not change the curve $C$,
  but ensures that $P'$ is positively oriented.
  This will be our standard convention.
\end{remark}

This elegant definition of $\Area(P)$ clarifies
the left hand side of Pick's equation.
For the right hand side $I + J/2 - 1$
we have to define --- and then count! ---
the enclosed lattice points.
To this end we use the winding number
of our polygon $P$ around some point $q \in \K^2$.

\subsection{The euclidean angle measure}

Over $\K = \R$, we can use the \emph{euclidean angle}
$\theta = \sphericalangle(u,v)$ between vectors $u,v \in \R^2$:
For each $u \ne 0$ and $v \notin u \, \R_{\le 0}$,
there exists a unique real number $\theta \in \ee{-\pi}{\pi}$
such that rotation by $\theta$ aligns $u$ with $v$:
\begin{align}
  \frac{v}{\abs{v}} = \begin{bmatrix}
    \cos\theta & -\sin\theta \\
    \sin\theta & \cos\theta
  \end{bmatrix} \frac{u}{\abs{u}} .
  \qquad
  \begin{tikzpicture}[x=20mm, y=20mm, baseline={(0.2,0.0)}]
    \draw[style= help lines, very thin, color=lightgray, step=0.1] (-1.001,-0.601) grid (1.001,0.801);
    \draw[style= help lines, very thin, color=gray, step=1.0] (-1.001,-0.601) grid (1.001,0.801);
    \draw[draw=none, fill=teal!30, fill opacity=0.3] (0,0) -- (30:0.5) arc (30:100:0.5) -- cycle;
    \draw[draw=black, -latex'] (30:0.5) arc (30:100:0.5);
    \draw (65:0.3) node[scale=0.8]{$\theta{>}0$};
    \draw[draw=none, fill=purple!30, fill opacity=0.3] (0,0) -- (30:0.5) arc (30:-40:0.5) -- cycle;
    \draw[draw=black, -latex'] (30:0.5) arc (30:-40:0.5);
    \draw (-5:0.3) node[scale=0.8]{$\theta'{<}0$};
    \draw[draw=blue, thick] (0,0) -- (30:0.9) (0,0) -- (100:0.7) (0,0) -- (-40:0.8);
    \draw[draw=black, fill=red] (0,0) circle[radius=2pt] node[shift={(-4pt,-5pt)}] {$0$};
    \filldraw[blue, radius=1pt]
    ( 30:0.9) circle node[below] {$u$} 
    (100:0.7) circle node[right, yshift=-5pt] {$v$}
    (-40:0.8) circle node[above, xshift=+5pt] {$v'$};
  \end{tikzpicture}
\end{align}

In the controversial case $v \in u \, \R_{<0}$,
both solutions $\theta = \pm\pi$ are equally possible.
Usually this case is ignored, forbidden, or arbitrated.
We democratically set $\sphericalangle(u,v) = 0$.
This choice may seem strange at first, but turns out to be advantageous.
It allows us to cover all cases uniformly,
and miraculously leads to the correct point count on the boundary.

\begin{definition}
  We define our \emph{euclidean angle measure}
  $\ang \colon \R^2 \times \R^2 \to \ee{-\sfrac{1}{2}}{+\sfrac{1}{2}}$ by
  \begin{align}
    \ang(u,v) := \begin{cases}
      \theta / 2\pi & \text{if $\abs{u} \cdot \abs{v} + u \cdot v > 0$,}
      \\ 0 & \text{if $\abs{u} \cdot \abs{v} + u \cdot v = 0$.}
    \end{cases}
  \end{align}
  
  By summing the angles of all edges of $P$,
  we obtain the (euclidean) \emph{winding number} 
  \begin{align}
    \Ang(P) := \sum_{i=1}^n \ang(p_{i-1},p_i) .
  \end{align}
\end{definition}

\begin{remark}
  If our polygon $P$ is closed and $0 \notin C$, then $\Ang(P) \in \Z$
  measures how often $P$ winds around the origin.
  Likewise $\Ang(P-q) \in \Z$ measures the winding number
  around the point $q \in \R^2 \minus C$,
  where $P-q = (p_0-q,p_1-q,\ldots,p_n-q)$.

  If $P$ is simply closed and positively oriented,
  then its Jordan decomposition $\R^2 = A \sqcup B \sqcup C$
  is characterized by the winding number:
  We have $\Ang(P-q) = 0$ for each exterior point $q \in A$
  and $\Ang(P-q) = 1$ for each interior point $q \in B$.
  We thus obtain
  \begin{align}
    I = \sum_{q \in \Z^2 \minus C} \Ang(P-q) .
  \end{align}
  
  Our careful definition pays off: 
  $\Ang(P-q)$ measures the interior angle at each boundary point $q \in C$.
  For $q \in \ee{p_{i-1}}{p_i}$ in the interior of any edge,
  we find $\Ang(P-q) = \sfrac{1}{2}$.  In each vertex, we find
  $\Ang(P-p_i) = \beta_i := \ang(p_{i+1}-p_i, p_{i-1}-p_i)$.
  By adding the turning angle $\alpha_i = \ang(p_i-p_{i-1}, p_{i+1}-p_i)$
  each vertex point is counted by $\sfrac{1}{2}$ as well.
  (This is illustrated in \autoref{fig:AreaAngle}.)
  The sum of all turning angles is $1$ by Hopf's umlaufsatz.

  If $P$ is a simply closed lattice polygon,
  with integer vertices $v_0,\ldots,v_n \in \Z^2$, then
  \begin{align}
    J/2 - 1 = \sum_{q \in \Z^2 \cap C} \Ang(P-q) .
  \end{align}
\end{remark}

The right hand side $I + J/2 - 1$ of Pick's equation
is defined geometrically, and usually formulated intuitively.
Any attempt to define it precisely relies on Jordan's theorem.
The winding number, as defined above, allows us to algebraically count
this as 
\begin{align}
  I + J/2 - 1  = \sum_{q \in \Z^2} \Ang(P-q) .
\end{align}

This reduces Pick's theorem to the following algebraic statement:

\begin{lemma}[Pick's lemma, using the euclidean angle measure] \label{lem:PickEuclidean}
  Let $p_1,\ldots,p_n=p_0 \in \Z^2$ be any sequence of integer points.
  Then the closed polygon $P = (p_0,p_1,\ldots,p_n)$ satisfies Pick's equation 
  \begin{align}
    \Area(P) = \sum_{q \in \Z^2} \Ang(P-q)
  \end{align}
\end{lemma}

Notice that this formulation does not require $P$ to be simple.
If $P$ is simply closed and positively oriented, 
then the right hand side equals $I + J/2 - 1$,
by Jordan and Hopf.

\subsection{The discrete angle measure}

The euclidean angle measure requires
real numbers and transcendental functions.
For computer implementation this is a heavy burden:
\texttt{Real} is good for proofs but bad for calculation,
whereas \texttt{Float} is good for approximations but bad for proofs.

For our purposes we prefer to work with the following \emph{discrete angle measure}.
This allows us to work over any ordered field $\K$,
for example $\Q \subseteq \K \subseteq \R$.

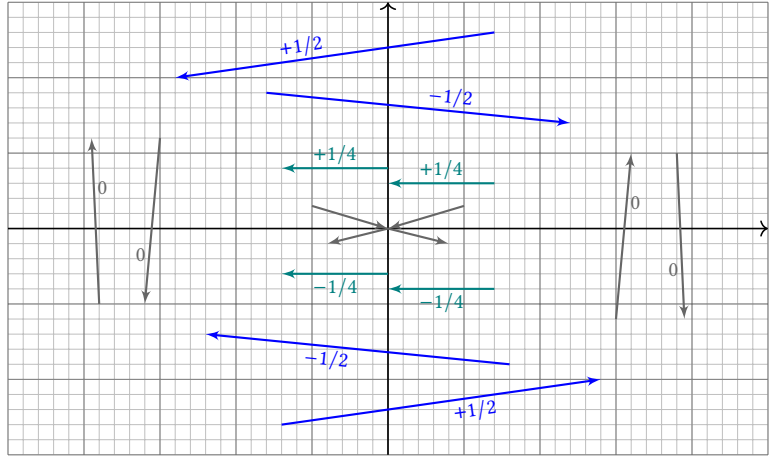
\begin{figure}[ht]
  \begin{tikzpicture}[x=10mm, y=10mm, inner sep=1pt, font=\footnotesize]
    \draw[style= help lines, very thin, color=lightgray, step=0.2] (-5,-3) grid (+5,+3);
    \draw[style= help lines, very thin, color=gray, step=1] (-5,-3) grid (+5,+3);
    \draw[semithick, ->] (-5,0) -- (+5,0); 
    \draw[semithick, ->] (0,-3) -- (0,+3); 
    \draw[thick, blue, -latex'] (+1.4,+2.6) -- (-2.8,+2.0)
    node[pos=0.6, sloped, above]{$+1/2$};
    \draw[thick, blue, -latex'] (-1.4,-2.6) -- (+2.8,-2.0)
    node[pos=0.6, sloped, below]{$+1/2$};
    \draw[thick, blue, -latex'] (+1.6,-1.8) -- (-2.4,-1.4)
    node[pos=0.6, sloped, below]{$-1/2$};
    \draw[thick, blue, -latex'] (-1.6,+1.8) -- (+2.4,+1.4)
    node[pos=0.6, sloped, above]{$-1/2$};
    \draw[thick, teal, -latex'] (+1.4,+0.6) -- (+0.0,+0.6)
    node[pos=0.5, sloped, above]{$+1/4$};
    \draw[thick, teal, -latex'] (-0.0,+0.8) -- (-1.4,+0.8)
    node[pos=0.5, sloped, above]{$+1/4$};
    \draw[thick, teal, -latex'] (+1.4,-0.8) -- (+0.0,-0.8)
    node[pos=0.5, sloped, below]{$-1/4$};
    \draw[thick, teal, -latex'] (-0.0,-0.6) -- (-1.4,-0.6)
    node[pos=0.5, sloped, below]{$-1/4$};
    \draw[thick, black!60, -latex'] (-3.0,+1.2) -- (-3.2,-1.0) node[pos=0.7, left]{$0$};
    \draw[thick, black!60, -latex'] (-3.8,-1.0) -- (-3.9,+1.2) node[pos=0.7, right]{$0$};
    \draw[thick, black!60, -latex'] (+3.0,-1.2) -- (+3.2,+1.0) node[pos=0.7, right]{$0$};
    \draw[thick, black!60, -latex'] (+3.8,+1.0) -- (+3.9,-1.2) node[pos=0.7, left]{$0$};
    \draw[thick, black!60, -latex'] (+1.0,+0.3) -- (+0.0,+0.0);
    \draw[thick, black!60, -latex'] (+0.0,+0.0) -- (+0.8,-0.2);
    \draw[thick, black!60, -latex'] (-1.0,+0.3) -- (-0.0,+0.0);
    \draw[thick, black!60, -latex'] (-0.0,+0.0) -- (-0.8,-0.2);
  \end{tikzpicture}
  \caption{The discrete angle measure}
  \label{fig:Dang}
\end{figure}

\begin{definition}
  For the edge between any two points $u,v \in \K^2$
  we define the \emph{discrete angle measure}
  as the number of axis crossings,
  as illustrated in Figure \ref{fig:Dang}:
  \begin{align}
    \dang(u,v) :=
    \frac{1}{4} \abs{ \sign u_1 - \sign v_1 }
    \cdot \sign \det \begin{bmatrix} u_1 & v_1 \\ u_2 & v_2 \end{bmatrix} .
  \end{align}

  By summing the angles of all edges of $P$,
  we obtain the (discrete) \emph{winding number} 
  \begin{align}
    \Dang(P) := \sum_{i=1}^n \dang(p_{i-1},p_i) .
  \end{align}
\end{definition}

If our polygon $P$ is closed and $0 \notin C$,
then $\Dang(P) = \Ang(P)$ measures the winding number around $0$.
The summands change, but the overall sum is the same;
the difference is only noticeable for boundary points $q \in C$ or open polygons.

The theorems of Jordan and Hopf, as stated above, continue
to hold with $\Dang(P)$ in place of $\Ang(P)$, see the Appendix.
We can thus define the elusive right hand side $I + J/2 - 1$ of Pick's equation
by the \emph{weighted sum of enclosed lattice points}
\begin{align}
  \label{eq:Welp}
  \Welp(P) := \sum_{q \in \Z^2} \Dang(P-q) .
\end{align}

\begin{lemma}[Pick's lemma, using the discrete angle measure] \label{lem:PickDiscrete}
  Let $p_1,\ldots,p_n=p_0 \in \Z^2$ be any sequence of integer points.
  Then the closed polygon $P = (p_0,p_1,\ldots,p_n)$ satisfies Pick's equation 
  \begin{align}
    \Area(P) = \Welp(P) .
  \end{align}
\end{lemma}

\begin{remark}
  Notice that this formulation does not require $P$ to be simple.
  If $P$ is simply closed and positively oriented, 
  then the right hand side equals $I + J/2 - 1$,
  by Jordan and Hopf.

  At first glance Pick's lemma may no longer look like Pick's classical theorem.
  Admittedly, the quantity $I + J/2 - 1$ is geometrically more intuitive,
  alas notoriously vague. (“Just look!”)
  Our formula for $\Welp(P)$ provides a precise definition,
  alas less intuitive. (“Just calculate!”)
  
  We cautiously call the above statement \emph{Pick's lemma}.
  In order to arrive at Pick's theorem,
  we have to invoke Jordan's decomposition and Hopf's umlaufsatz
  for simply closed polygonal curves.
  This bridges the gap between the algebraic lemma
  and the geometric theorem.

  Since Jordan's and Hopf's are classical theorems in their own right,
  we do not consider them as part of the \emph{proof} of Pick's lemma,
  but rather as the foundation to \emph{motivate} or \emph{interpret} the statement.
  This groundwork being provided, we can now focus on proving Pick's lemma.
\end{remark}

\section{Proof of Pick's lemma}

The sum \eqref{eq:Welp} over $q \in \Z^2$ has finite support:
We can restrict it to a sufficiently large
square box $Q = \{-r,\ldots,r\}^2 \subset \Z^2$
containing all vertices $p_1,\ldots,p_n$.
Each point $q \in \Z^2 \minus Q$ yields $\Dang(P-q) = 0$.
Using this restriction, we can swap the double sum: 
\begin{alignat}{3}
  & \Welp(P)
  && := \sum_{q \in \Z^2} \Dang(P-q)
  && = \sum_{q \in Q} \sum_{i=1}^n \dang(p_{i-1}-q, p_i-q) 
  \\ \notag & && && = \sum_{i=1}^n
  \underbracket{ \sum_{q \in Q} \dang(p_{i-1}-q, p_i-q) }_{\normalsize =: \welp(p_{i-1},p_i)}
\end{alignat}

Now both sides of Pick's equation look formally very similar:
\begin{alignat}{3}
  & \Area(P) && = \sum_{i=1}^{n} \area(p_{i-1},p_i) ,
  \\
  & \Welp(P) && = \sum_{i=1}^n \welp(p_{i-1},p_i) .
\end{alignat}

Here another miracle happens:
Both sums are not only equal, but termwise equal!

For any to lattice points $u,v \in Q$
in our square box $Q = \{-r,\ldots,r\}^2 \subset \Z^2$
we show that
\begin{align}
  \area(u,v) = \welp(u,v) .
\end{align}

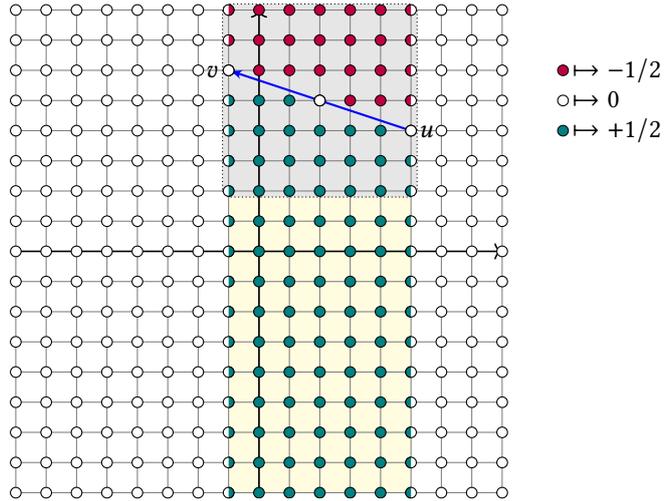
\begin{figure}[ht]
  \begin{tikzpicture}[x=4mm, y=4mm, font=\mathstrut]
    \draw[draw=none] (-16,-8) rectangle (+16,+8);
    \draw[draw=none, fill=yellow!50, fill opacity=0.3]
    (-1.0,-8.0) rectangle (+5.0,+2.0);

    \draw[draw=black, densely dotted, rounded corners=1mm, fill=black!10]
    (-1.2,+1.8) rectangle (+5.2,+8.2);
    \draw[color=black!50, step=1] (-8,-8) grid (+8,+8);
    \draw[semithick, ->] (-8,0) -- (+8,0); 
    \draw[semithick, ->] (0,-8) -- (0,+8); 
    
    \filldraw[draw=blue, thick, -latex'] (+5,+4) -- (-1,+6);
    \draw[draw=black, fill=white, radius=2.0pt]
    (+5,+4) circle node[right] {$u$}
    (-1,+6) circle node[left] {$v$}
    (+2,+5) circle;
    
    \foreach \x in {-8,-7,...,-2,6,7,...,8}{
      \foreach \y in {-8,...,8}{
        \draw[draw=black, fill=white] (\x,\y) circle[radius=2.0pt];
      }
    }
    
    \draw[draw=black, fill=white] (10,5) circle[radius=2.0pt];
    \draw (10,5) node[right] {$\mapsto 0$};
    
    \begin{scope}
      \foreach \y in {7,...,8}{
        \draw[draw=black, fill=white] (-1,\y) circle[radius=2.0pt];
      }
      \foreach \y in {5,...,8}{
        \draw[draw=black, fill=white] (+5,\y) circle[radius=2.0pt];
      }
      \clip (+5,+4.0) -- (+5,+8.2) -- (-1,+8.2) -- (-1,+6.0) -- cycle;
      \foreach \x in {-1,...,5}{
        \foreach \y in {-8,...,8}{
          \draw[draw=black, fill=purple] (\x,\y) circle[radius=2.0pt];
        }
      }
    \end{scope}
    
    \draw[draw=black, fill=purple] (10,6) circle[radius=2.0pt];
    \draw (10,6) node[right] {$\mapsto -1/2$};
    
    \begin{scope}
      \foreach \y in {-8,...,5}{
        \draw[draw=black, fill=white] (-1,\y) circle[radius=2.0pt];
      }
      \foreach \y in {-8,...,3}{
        \draw[draw=black, fill=white] (+5,\y) circle[radius=2.0pt];
      }
      \clip (-1,-8.2) -- (+5,-8.2) -- (+5,+4.0) -- (-1,+6.0) -- cycle;
      \foreach \x in {-1,...,5}{
        \foreach \y in {-8,...,8}{
          \draw[draw=black, fill=teal] (\x,\y) circle[radius=2.0pt];
        }
      }
    \end{scope}
    
    \draw[draw=black, fill=white]
    (-1,6) circle[radius=2.0pt]
    (+2,5) circle[radius=2.0pt]
    (+5,4) circle[radius=2.0pt];
    
    \draw[draw=black, fill=teal] (10,4) circle[radius=2.0pt];
    \draw (10,4) node[right] {$\mapsto +1/2$};
  \end{tikzpicture}
  \caption{Proving $\welp(u,v) = \area(u,v)$ by inspecting $q \mapsto \dang(u-q,v-q)$}
  \label{fig:AreaEqualsWelp}
\end{figure}


The proof is illustrated in \autoref{fig:AreaEqualsWelp}
by plotting $q \mapsto \dang(u-q,v-q)$.
The algebraic calculation proceeds as follows.
We assume $v_1 < u_1$ and $u_2 + v_2 \ge 2$;
the other cases are symmetric.
\begin{enumerate}
\item
  For $q_1 < v_1$ or $q_1 > u_1$ we have $\dang(u-q,v-q) = 0$.
\item
  The rectangle $R = \{v_1,\ldots,u_1\} \times \{u+v-r,\ldots,r\}$
  allows the involution $q \mapsto u+v-q$.
  We find $\dang(u-(u+v-q), v-(u+v-q)) = \dang(v-q,u-q) = -\dang(u-q,v-q)$.
  Thus all summands cancel pairwise,
  and we obtain $\sum_{q \in R} \dang(u-q,v-q) = 0$.
\item
  The remaining summands add up to $\frac{1}{2} (u_1-v_1) (u_2+v_2) = \area(u,v)$.
  More precisely, we find $\dang(u-q,v-q) = \sfrac{1}{2}$
  for $q \in \{u_1+1,\ldots,v_1-1\} \times \{-r,\ldots,u_2+v_2-r-1\}$
  and $\dang(u-q,v-q) = \sfrac{1}{4}$
  for $q \in \{u_1,v_1\} \times \{-r,\ldots,u_2+v_2-r-1\}$.
\end{enumerate}

This proves $\welp(u,v) = \area(u,v)$
for any edge with end points $u,v \in Q$.
By summing over all edges of our polygon $P$,
we conclude $\Area(P) = \Welp(P)$, as claimed. 


\begin{remark}
  In this algebraic form, the proof is readily implemented in a proof assistant.
  For a proof in Lean see \url{https://github.com/Palamedez314/PicksTheorem2025}.
  This completes the promised formalization.
\end{remark}

\appendix

\setcounter{section}{0}
\section{Axiomatic angle measure}

So far we have focused on Pick's lemma as our central contribution.
Two more classical theorems complete the picture:
First, we need Jordan's decomposition to formulate Pick's theorem precisely, 
translating it from geometric to algebraic terms.
Second, we add Hopf's umlaufsatz to interpret the counting result 
in a more geometric and user-friendly fashion.

\begin{figure}[ht]
  \begin{tikzpicture}[x=10mm, y=10mm]
    \draw (-3,+1) node[draw=black, inner sep=3pt, font=\strut]{polygons};
    \draw (+3,+1) node[draw=black, inner sep=3pt, font=\strut]{counting};
    \draw (-3,-1) node[draw=black, inner sep=3pt, font=\strut]{winding};
    \draw (+3,-1) node[draw=black, inner sep=3pt, font=\strut]{summing};
    \draw[-latex'] (-2.1,+1.0) -- (+2.1,+1.0) node[midway, above]{Pick's theorem};
    \draw[-latex'] (-2.1,-1.0) -- (+2.1,-1.0) node[midway, above]{Pick's lemma};
    \draw[-latex'] (-3.0,+0.6) -- (-3.0,-0.6) node[midway, left]{Jordan};
    \draw[-latex'] (+3.0,-0.6) -- (+3.0,+0.6) node[midway, right]{Hopf};
  \end{tikzpicture}
  \caption{We transform Pick's Theorem to Pick's lemma}
  \label{fig:PickTransform}
\end{figure}
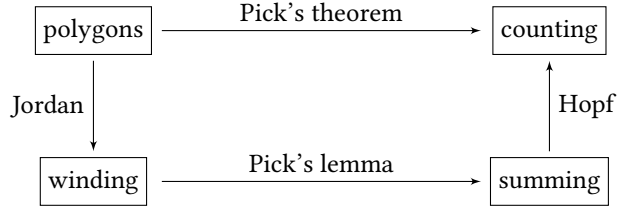


We can use different angle measures for the theorems of
Jordan and Hopf and finally Pick.
In order to cover all cases simultaneously, 
we extract the essential properties used in the proofs:

\begin{definition}[angle measure\label{def:AngleMeasure}]
  Let $(\K,+,\cdot,\le)$ be an ordered field.
  An \emph{angle measure} is a map $\mu \colon \K^2 \times \K^2 \to \K$
  satisfying the following conditions:
  \begin{enumerate}
  \item
    Scaling: For all $u,v \in \K^2$ and $\lambda \in \K_{>0}$
    we have $\mu(u,v) = \mu(\lambda u, v) = \mu(u, \lambda v)$.
  \item
    Symmetry: $\mu(v,u) = -\mu(u,v)$ and $\mu(-u,-v) = \mu(u,v)$.
  \item
    Addition: $\mu(u,v) = \mu(u,s) + \mu(s,v)$ for $s \in [u,v]$.
  \item
    Normalization: $\mu(e_1,e_2) = \sfrac{1}{4}$.
  \end{enumerate}
\end{definition}

Examples include the euclidean angle measure $\ang$
and the discrete angle measure $\dang$.

\setcounter{section}{9}
\section{Jordan's decomposition for polygons}

Jordan's decomposition for polygons can be formulated
and proven using any angle measure. 
We work over any ordered field $(\K,+,\cdot,\le)$
and continue to use the notation of the introduction.
Let $P = (p_0,p_1,\ldots,p_n)$ be a polygon in $\K^2$
with corresponding path $\gamma \colon [0,1] \to \K^2$
and image curve $C = \gamma([0,1]) \subset \K^2$.
Using the angle measure $\mu$, we define the winding number 
\begin{align*}
  \mu(P) := \sum_{i=1}^n \mu(p_{i-1},p_i) .
\end{align*}

\begin{theorem}
  Let $P = (p_0,p_1,\ldots,p_n)$ be a polygon in $\K^2$,
  simply closed and positively oriented.
  
  (1) The complement $X := \K^2 \minus C = A \sqcup B$
  decomposes into two connected components, namely  
  \begin{align*}
    A & = \set{ q \in X | \mu(P-q) = 0 } ,
    \\
    B & = \set{ q \in X | \mu(P-q) = 1 } .
  \end{align*}

  (2) The curve $C$ is the boundary of $A$ and of $B$ within the plane $\K^2$.

  (3) At each boundary point $q \in C$, the value $\mu(P-q)$ measures the interior angle:

  (3a) At each point $q \in \ee{p_{i-1}}{p_i}$ in the interior of an edge,
  we find $\mu(P-q) = \sfrac{1}{2}$.
  
  (3b) At each vertex $q = p_i$ we find
  $\mu(P-p_i) = \beta_i := \mu(p_{i+1}-p_i, p_{i-1}-p_i)$.
\end{theorem}

\begin{remark}
  We assume positive orientation, so that $\area(P) > 0$.
  Geometrically this ensures that, when traveling along the path $\gamma$,
  the interior region $B$ will be on the left of our curve $C$.
  
  (1) The usual topological notions over $\R$
  now acquire an algebraic flavour over $\K$:
  Two points $a,b \in X$ are \emph{connected} in $X$,
  if they can be connected by a polygonal path within $X$.
  
  (2) The fact that $C$ is the common boundary means that
  each point $q \in C$ has arbitrarily close neighbours in $A$ and in $B$.
  These can be reached by the two bisectors of the angle at $q$.

  (3) The extension to boundary points $q \in C$ 
  follows from (1) by cutting the corner at $q$. 
  It is not only oddly satisfying but also practically useful.
  In particular it leads us to Pick's theorem. 
\end{remark}

A pen-and-paper proof 
can be established as in the classical setting,
and then implemented in a proof assistant. 
This foundational work usually requires
“augmenting the existing geometry libraries” \cite{Binder:Kosaian:2024}.
It merits further exploration beyond the scope of the modest article.

\setcounter{section}{7}
\section{Hopf's umlaufsatz for polygons}

Hopf's umlaufsatz for polygons can be formulated
and proven using any angle measure $\mu$.

\begin{theorem}
  Let $P = (p_0,p_1,\ldots,p_n)$ be a simply closed
  polygon in $\K^2$ and positively oriented.
  
  At each vertex $p_i$, cyclically indexed
  by integers modulo $n$, we define the turning angle
  \begin{align*}
    \alpha_i := \mu(p_{i+1}-p_i, p_i-p_{i-1}) .
  \end{align*}

  Then the total tangent turning number (or umlaufzahl)
  adds up to $\sum_{i=1}^n \alpha_i = 1$.
\end{theorem}

\begin{remark}
  The turning angle $\alpha_i$ and the interior angle $\beta_i$
  add up to $\sfrac{1}{2}$.  Summing over all vertices, we obtain
  $n/2 = \sum_{i=1}^n \alpha_i + \beta_i = 1 + \sum_{i=1}^n \beta_i$,
  whence $\sum_{i=1}^n \beta_i = n/2 - 1$.

  The simplest case $n=3$ is the familiar statement that,
  in any triangle, the interior angles add up
  to $\beta_1 + \beta_2 + \beta_3 = 1/2$,
  corresponding to $\pi/2$ or $180^\circ$.
\end{remark}

Together with Jordan's theorem, we obtain the desired
point count $\sum_{q \in \Z^2 \cap C} \mu(P-q) = J/2 - 1$.

\subsection{Formalizing Hopf's umlaufsatz}

A proof of Hopf's umlaufsatz is straight-forward for triangles.
It immediately extends to convex polygons and
more generally to star-shaped polygons.
It seems intuitively plausible for any simply closed polygon,
but in this general case a proof is far from obvious.
It becomes false for non-simply closed polygons.

\begin{figure}[ht]
  \begin{tikzpicture}[x=5mm, y=5mm, font=\footnotesize\mathstrut]
    \begin{scope}[shift={(0,1)}]
      \draw[line width=1pt, blue, fill=yellow!50, fill opacity=0.3]
      (0,2) -- (2,0) -- (5,0) -- (5,2) -- (4,3) -- (5,5) --
      (2,6) -- (0,3) -- (2,5) -- (4,1) -- (3,1) -- (2,3) -- cycle;
      \draw[inner sep=1pt, font=\mathstrut] 
      (2,0) node[below left]{$p_{n-2}$}
      (3,0) node[below]{$p_{n-1}$}
      (4,0) node[below right]{$p_0$}
      (5,0) node[below right]{$p_1$};
      \foreach \x/\y in { 0/2, 2/0, 3/0, 4/0, 5/0, 5/2, 4/3, 5/5, 2/6, 0/3, 2/5, 4/1, 3/1, 2/3 }{
        \draw[fill=black] (\x,\y) circle[radius=1pt];
      }
    \end{scope}
    \begin{scope}[shift={(9,0)}]
      \draw[draw=black!30, step=1.0] (0,1) grid (5,6);
      \draw[-latex'] (0,1) -- (6,1) node[above, xshift=-4pt]{$i$};
      \draw[-latex'] (0,1) -- (0,7) node[right, yshift=-4pt]{$j$};
      \draw (0,1) node[below]{$0$} (2.5,1) node[below]{$\ldots$} (5,1) node[below]{$n{-}2$};
      \draw (0,1) node[left]{$1$}  (0,3.5) node[left]{$\vdots$}  (0,6) node[left]{$n{-}1$};
      \foreach \i [
        evaluate= \i as \ii using {int(\i+2)}
      ] in {0,...,3}{
        \foreach \j in {\ii,...,5}{
          \draw[draw=black, fill=teal!10] (\i,\j) rectangle (\i+1,\j+1);
        }
      }
      \foreach \i in {0,...,4}{
        \draw[draw=black, fill=blue!10] (\i,\i+1) |- (\i+1,\i+2) -- cycle;
      }
      \foreach \i/\j in {0/1, 1/2, 2/3, 3/4, 4/5, 5/6, 4/6, 3/6, 2/6, 1/6, 0/6, 0/5, 0/4, 0/3, 0/2}{
        \draw[fill=black] (\i,\j) circle[radius=1pt];
      }
      \draw (3.5,3.5) node[right, align=center]{
        discrete homotopy \\ from the diagonal \\ to both sides};
    \end{scope}
  \end{tikzpicture}
  \caption{Hopf's homotopy argument for his umlaufsatz}
  \label{fig:HopfTheorem}
\end{figure}
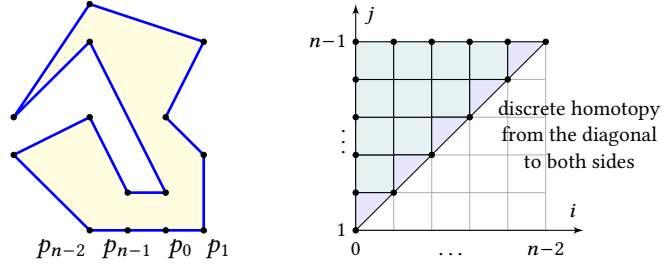

Hopf first formulated his umlaufsatz \cite{Hopf:1935} for simply closed paths
$\gamma \colon [0,1] \to \R^2$ that are continuously differentiable,
and proved it by an ingenious homotopy argument.
%
%
He then showed that it suffices to require
a \emph{piecewise} continuously differentiable curve.
For future reference, we sketch an algebraic version
for polygons over any ordered field $(\K,+,\cdot,\le)$,
using any angle measure $\mu$.
The key idea is adapted from Hopf's original
and illustrated in \autoref{fig:HopfTheorem}.

\begin{proof}[Hopf's homotopy proof]
  We consider the vertices $p_1,\ldots,p_n=p_0$ with $p_i = (x_i,y_i) \in \K^2$.
  After renumbering, we can assume that $y_0$ is minimal.
  After translation we have $p_0 = (0,0)$ and $y_i \ge 0$ for all $i$.
  By inserting small edges around $p_0$,
  we can achieve $0 = y_0 = y_1 = y_{n-1} = y_{n-2}$
  and $x_{n-2} < x_{n-1} < x_0 < x_1$.
  (This is shown on the left of \autoref{fig:HopfTheorem}.)
  
  For $0 \le i < j \le n-1$ we consider the secant $s_{i,j} := p_j - p_i$.
  From $s_{0,1}$ via $s_{0,j}$ to $s_{0,n-1}$ we perform a half turn,
  then via $s_{i,n-1}$ to $s_{n-2,n-1}$ we perform another half turn.
  Both add up to $1$.  (This uses that our polygon is simple,
  a degenerate polygon with $0 = y_0 = \ldots = y_n$ yields $0$.)
  
  If, moreover, our polygon $P$ is simply closed,
  then this movement corresponds to the umlaufzahl.
  More precisely, we prove the following equation by a discrete homotopy:
  \begin{align}
    \label{eq:Hopf}
    \sum_{k=1}^{n-2} \mu( s_{k-1,k}, s_{k,k+1} )
    \;=\;
    \sum_{j=1}^{n-2} \mu( s_{0,j}, s_{0,j+1} ) +
    \sum_{i=1}^{n-2} \mu( s_{i-1,n-1}, s_{i,n-1} )
  \end{align}
  \begin{enumerate}
  \item
    Triangles: For $0 \le i \le n-2$ and $j = i+1$ we find
    \begin{align*}
      \mu( s_{i-1,i}, s_{i,i+1} ) =
      \mu( s_{i-1,i}, s_{i-1,i+1} ) +
      \mu( s_{i-1,i+1}, s_{i,i+1} ) .
    \end{align*}
  \item
    Squares: For $0 \le i < n-2$ and $i+1 < j \le n-1$ we find
    \begin{align*}
      0 = \mu( s_{i,j}, s_{i,j+1} ) + \mu( s_{i,j+1}, s_{i+1,j+1} )
      - \mu( s_{i,j}, s_{i+1,j} ) - \mu( s_{i+1,j}, s_{i+1,j+1} ) .
    \end{align*}
    This essentially uses that our polygon is simple,
    so $[p_i,p_{i+1}] \cap [p_j,p_{j+1}] = \emptyset$.
  \end{enumerate}

  The left hand side of \eqref{eq:Hopf}
  is the umlaufzahl. 
  We replace each summand using (1) and then add all sums from (2).
  Most of the summands cancel pairwise,
  what remains is the right hand side of \eqref{eq:Hopf}.
  This equals $1$, as explained above.
\end{proof}

\setcounter{section}{22}
\section{A “water proof” plausibility argument}

In every-day mathematical communication we often only sketch
the idea and appeal to some degree of informal intuition.
This is especially true for geometric statements,
like Pick's theorem, where we often heavily rely on our visual perception.
This is usually a good thing for human learning and understanding,
but may also inhibit the desire for a sound formalization.

We sketch such a visual proof, appealing to geometric intuition
as a counterpart to our algebraic formalization.
We then explain how the intuition can be made precise
and finally leads to our algebraic viewpoint,
closing this circle of ideas.

\begin{figure}[ht]
  \begin{tikzpicture}[x=5mm, y=5mm, rounded corners=0.1pt]
    \begin{scope}[shift={(0,0)}]
      \draw[color=black!50, step=1] (-1,-1) grid (6,7);
      \draw[line width=1pt, black] 
      (0,2) -- (2,0) -- (5,0) -- (5,2) -- (4,3) -- (5,5) --
      (2,6) -- (0,3) -- (2,5) -- (4,1) -- (3,1) -- (2,3) -- cycle;
      \begin{scope}
        \clip (-1,-1) rectangle (6,7)
        (0,2) -- (2,0) -- (5,0) -- (5,2) -- (4,3) -- (5,5) --
        (2,6) -- (0,3) -- (2,5) -- (4,1) -- (3,1) -- (2,3) -- cycle;
        \foreach \x in {-1,0,...,6}{
          \foreach \y in {-1,0,...,7}{
            \draw[draw=none, fill=black!30, fill opacity=0.4] (\x,\y) circle[radius=1mm];
          }
        }
      \end{scope}
      \begin{scope}
        \clip (0,2) -- (2,0) -- (5,0) -- (5,2) -- (4,3) -- (5,5) --
        (2,6) -- (0,3) -- (2,5) -- (4,1) -- (3,1) -- (2,3) -- cycle;
        \foreach \x in {-1,0,...,6}{
          \foreach \y in {-1,0,...,7}{
            \draw[draw=none, fill=blue!60, fill opacity=0.4] (\x,\y) circle[radius=1mm];
          }
        }
      \end{scope}
    \end{scope}
    \begin{scope}[shift={(9,0)}]
      \draw[color=black!50, step=1] (-1,-1) grid (6,7);
      \draw[line width=1pt, black] 
      (0,2) -- (2,0) -- (5,0) -- (5,2) -- (4,3) -- (5,5) --
      (2,6) -- (0,3) -- (2,5) -- (4,1) -- (3,1) -- (2,3) -- cycle;
      \begin{scope}
        \clip (-1,-1) rectangle (6,7)
        (0,2) -- (2,0) -- (5,0) -- (5,2) -- (4,3) -- (5,5) --
        (2,6) -- (0,3) -- (2,5) -- (4,1) -- (3,1) -- (2,3) -- cycle;
        \foreach \x in {-1,0,...,6}{
          \foreach \y in {-1,0,...,7}{
            \draw[draw=none, fill=black!30, fill opacity=0.2] (\x,\y) circle[radius=2mm];
          }
        }
      \end{scope}
      \begin{scope}
        \clip (0,2) -- (2,0) -- (5,0) -- (5,2) -- (4,3) -- (5,5) --
        (2,6) -- (0,3) -- (2,5) -- (4,1) -- (3,1) -- (2,3) -- cycle;
        \foreach \x in {-1,0,...,6}{
          \foreach \y in {-1,0,...,7}{
            \draw[draw=none, fill=blue!60, fill opacity=0.2] (\x,\y) circle[radius=2mm];
          }
        }
      \end{scope}
    \end{scope}
    \begin{scope}[shift={(18,0)}]
      \draw[color=black!50, step=1] (-1,-1) grid (6,7);
      \draw[line width=1pt, black] 
      (0,2) -- (2,0) -- (5,0) -- (5,2) -- (4,3) -- (5,5) --
      (2,6) -- (0,3) -- (2,5) -- (4,1) -- (3,1) -- (2,3) -- cycle;
      \begin{scope}
        \clip (-1,-1) rectangle (6,7)
        (0,2) -- (2,0) -- (5,0) -- (5,2) -- (4,3) -- (5,5) --
        (2,6) -- (0,3) -- (2,5) -- (4,1) -- (3,1) -- (2,3) -- cycle;
        \foreach \x in {-1,0,...,6}{
          \foreach \y in {-1,0,...,7}{
            \draw[draw=none, fill=black!30, fill opacity=0.1] (\x,\y) circle[radius=4mm];
          }
        }
      \end{scope}
      \begin{scope}
        \clip (0,2) -- (2,0) -- (5,0) -- (5,2) -- (4,3) -- (5,5) --
        (2,6) -- (0,3) -- (2,5) -- (4,1) -- (3,1) -- (2,3) -- cycle;
        \foreach \x in {-1,0,...,6}{
          \foreach \y in {-1,0,...,7}{
            \draw[draw=none, fill=blue!80, fill opacity=0.1] (\x,\y) circle[radius=4mm];
          }
        }
      \end{scope}
    \end{scope}
  \end{tikzpicture}
  \caption{\textgreek{Πάντα ῥεῖ}, everything flows.  We let the water do the work.}
  \label{fig:WaterFlow}
\end{figure}
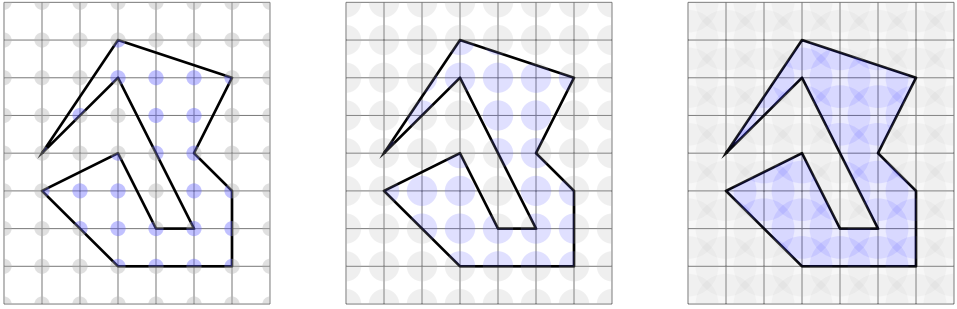

\begin{proof}[The water proof]
  At each integer point $q \in \Z^2$ we place a unit drop of water.
  The water then flows evenly in all directions.  Finally, the plane is
  uniformly covered with water, one unit of water per unit square.
  (To ensure finiteness, we can think of periodic boundary conditions.)

  Now consider two integer points $u,v \in \Z^2$.
  Half rotation about its center $c = (u+v)/2$
  defines the involution $\rho \colon z \mapsto u + v - z$.
  This reverses the edge and thus the flow.
  Since $\rho$ maps the lattice $\Z^2$ to itself,
  the water that flows from the point $q \in \Z^2$ over $[u,v]$ 
  is compensated by the water flowing from the point $\rho(q) \in \Z^2$ over $[u,v]$.
  The net flow over the edge $[u,v]$ is zero.
  
  This shows that the total amount of water within $B$ never changes.
  At the start it is the share of drops that fall into $B$.
  At the end it equals the surface area of $B$.
  We thus obtain:
  \begin{align*}
    \vol_2(B) = I + \sum_{q \in \Z^2 \cap C} \beta(q) / 2\pi
  \end{align*}

  Here $\beta(q)$ is the interior angle at the boundary point $q \in C$.
  So our physical intuition, or gedankenexperiment, leads us to
  Pick's lemma \ref{lem:PickEuclidean} using the euclidean angle measure.

  We once again invoke Hopf's umlaufsatz:
  the sum of turning angles is $\sum_{q \in C} [ \pi - \beta(q) ] = 2\pi$,
  so we arrive at Pick's theorem $\vol_2(B) = I + J/2 - 1$.
\end{proof}



\begin{remark}
  This “water proof” is a wonderous example
  of a physical plausibility argument.
  To many people is seems convincing
  because it appeals to our physical experience
  and our tried-and-tested understanding of the world.
  But is it a proof, really?
  Can we actually understand the physical world?
  It is easy to visualize, intuitively, but hard to formalize, rigorously.
\end{remark}

We can give the physical argument more mathematical substance.
First, we want to ensure finiteness of our experimentation table,
so instead of the entire plane $\R^2$ we consider a square $[-r,+r]^2$
with sufficiently large $r \in \N$ and periodic boundary conditions.
Geometrically, this is a flat torus $T = \R^2 / 2r\Z^2$,
and we keep all necessary symmetries: translations and reflections.

Second, we specify the properties of our idealized water drops.
For each time $t \in \ei{0}{1}$ let $\mu_t$ be a continuous probability measure,
ending with the uniform (Haar) measure $\mu_1$ on $T$,
so that $4 r^2 \mu_1(B)$ is the area of our inner region. 
Now we observe the quantity 
\begin{align*}
  f(t) := \sum_{q \in \Z^2/2r\Z^2} \mu_t(B-q) .
\end{align*}

We have $f(1) = \vol_2(B)$.
For each $t > 0$ we assume $\mu_t$ to be symmetric
with respect to $z \mapsto -z$ and uniform on its support.
Assuming that $\mu_t$ varies continuously with $t$,
the above symmetry argument suggests that $f$ is constant.
At the other end we look at the limit for $t \searrow 0$:

\begin{example}
  As in \autoref{fig:WaterFlow}, for small $t$ 
  we let $\supp \mu_t = D(0,t)$ be the disc of radius $t$ around the origin $0$.
  For sufficiently small $t$ we obtain $f(t) = \sum_{q \in \Z^2} \Ang(P-q)$
  as in \ref{lem:PickEuclidean}.

\end{example}

\begin{example}
  Alternatively, for small $t$ 
  we consider $\supp \mu_t = [-t^2,t^2] \times [-t,t]$
  to be a thin vertical rectangle.
  In the limit $t \searrow 0$ we obtain
  $f(t) \to \sum_{q \in \Z^2} \Dang(P-q)$
  as in \ref{lem:PickDiscrete}.
  
\end{example}

In this fashion our physical intuition can lead us to both flavours of Pick's lemma.

\begin{remark}
  This tale can be recounted by replacing water with heat.
  This is appealing for a physically inclined audience
  familiar with the heat equation. 
  Moreover, this provides an explicit model of the flow
  as partial differential equations, this time for $t \in \ee{0}{\infty}$.
  In this model 
  we can then prove that $f$ is indeed constant
  and calculate the limit for $t \searrow 0$ and for $t \nearrow \infty$.
\end{remark}


\setcounter{section}{16}
\section{Quotes}

Pick's theorem has been formally implemented, with considerable effort,
first in 2011 by John Harrison in \texttt{HOL\,Light} \cite{Harrison:2011},
then in 2024 by Sage Binder and Katherine Kosaian in \texttt{Isabelle}
\cite{Binder:Kosaian:2024}.

\begin{quote}
  The proof turned out to be somewhat harder work than expected. [\dots]
  Charming and surprising as Pick's theorem is, one would normally consider it
  a very straight-forward result [\dots].
  What accounts for the difficulty of formalizing Pick?
  Of course, it could be just a reflection of our own ineptitude.
  However we prefer to believe that there is an intrinsic difficulty
  in formalizing many geometric proofs in stating and proving
  in a formal way various properties that seem intuitively obvious ‘by eye’.
  (John Harrison \cite{Harrison:2011})
\end{quote}

We close with a heartwarming anecdote concerning practical applicability: 

\begin{quote}
  Some years ago, the Northwest Mathematics Conference was held in Eugene,
  Oregon. To add a bit of local flavor, a forester was included on the program,
  and those who attended his session were introduced to a variety of nice
  examples which illustrated the important role that mathematics plays in the
  forest industry. One of his problems was concerned with the calculation of
  the area inside a polygonal region drawn to scale from field data obtained for
  a stand of timber by a timber cruiser. The standard method is to overlay a
  scale drawing with a transparency on which a square dot pattern is printed.
  Except for a factor dependent on the relative sizes of the drawing and the
  square grid, the area inside the polygon is computed by counting all of the
  dots fully inside the polygon, and then adding half of the number of dots
  which fall on the bounding edges of the polygon. Although the speaker was
  not aware that he was essentially using Pick's formula, I was delighted to see
  that one of my favorite mathematical results was not only beautiful, but even
  useful. (Duane DeTemple, cited in \cite{Gruenbaum:Shepard:1993})
\end{quote}

\section*{Remerciements}

Elli zu Neblberg expresses her heartfelt gratitude to all of her collaborators.


\bibliographystyle{amsplain}
\bibliography{Pick-in-Lean}

\end{document}